\documentclass[11pt,reqno]{amsart}
\usepackage{amsfonts,latexsym}
\usepackage{amsmath}
\usepackage{amscd}
\usepackage{float,amsmath,amssymb,mathrsfs,bm,multirow,graphics}
\usepackage[dvips]{graphicx}
\usepackage{amssymb}
\usepackage{amscd}
\usepackage[utf8]{inputenc}
\usepackage[english]{babel}
 
\usepackage{amsthm}

\def\b{\begin{eqnarray}}
\def\e{\end{eqnarray}}

\newcommand{\T}{{\Bbb T}}
\newcommand{\N}{{\Bbb N}}
\newcommand{\C}{{\Bbb C}}

\newcommand{\Z}{{\Bbb Z}}

\newtheorem{theorem}{Theorem}[section]
\newtheorem{proposition}[theorem]{Proposition}
\newtheorem{lemma}[theorem]{Lemma}

\newtheorem{remark}[theorem]{Remark}
\newtheorem{corollary}[theorem]{Corollary}

\begin{document}
\title[Big Hankel operators on vector-valued Fock spaces in $\C^d$]{Big Hankel operators on vector-valued Fock spaces in $\C^d$}

\author{H\'el\`ene Bommier-Hato}

\address{Bommier-Hato: Faculty of Mathematics,
University of Vienna,
Oskar-Morgenstern-Platz 1, 1090 Vienna, Austria\  \ \& \ Institut de Math\'ematiques de Marseille, Universit\'e de Provence,  13453 Marseille Cedex 13, France} \email{helene.bommier@gmail.com}
 

\author{Olivia Constantin}
\address{
Constantin: Faculty of Mathematics,
University of Vienna,
Oskar-Morgenstern-Platz 1, 1090 Vienna, Austria
Austria}
\email{olivia.constantin@univie.ac.at}

\thanks{The authors were supported by the FWF project P 30251-N35.}
\date{}




\begin{abstract} We study big Hankel operators acting on vector-valued Fock spaces with radial weights in $\C^d$. We provide complete characterizations 
for the boundedness, compactness and Schatten class membership of such operators. 
\end{abstract}

\maketitle

\section{Introduction}

The classical Fock space (or, the Segal-Bargmann space) 
has a long  and celebrated history and its origins are found in quantum mechanics. 
In spite of the richness of the existing literature on scalar Fock spaces,  the vector-valued case
has, to the best of our knowledge, not yet been thoroughly considered.  
The investigation of spaces of analytic functions in the vector-valued framework
brings along new insights
 and it often requires the development 
of entirely new techniques compared to the scalar setting (see \cite{Peller}).
The objective of our paper is to study big Hankel operators with anti-analytic symbols on generalized vector-valued Fock spaces.

Seip and Youssfi \cite{SY} studied big Hankel operators with anti-holomorphic symbols acting on a large class of scalar
Fock spaces with radial weights subject to a mild smoothness condition (see below). Using their sharp estimates for the reproducing kernel,  we investigate this class of operators 
in the vector-valued setting and define  adequate versions of  Bloch, Besov spaces and of mean oscillation.
 
 Let us first present our framework. 
We assume $\Psi: [0,\infty) \to [0,\infty)$ is a $C^3$-function such that
\begin{equation}\label{psi}
\Psi'(x) >0\,,\quad \Psi''(x) \ge 0\quad\text{and}\quad \Psi'''(x) \ge 0\,.
\end{equation}
 We now define the class $\mathcal{S}$ of functions $g: [0,\infty) \to [0,\infty)$ such that there exists a real number $\eta<\frac{1}{2}$  for which
\begin{equation}\label{phi 1}
g''(x)=O\left(x^{-\frac{1}{2}}\left[g'(x)\right]^{1+\eta}\right),\ x\rightarrow+\infty.
\end{equation}
We assume that the function
$$\Phi(x):=x \Psi'(x) $$
is in $\mathcal{S}$, and, when $d>1$, we also require that $\Psi$ is in $\mathcal{S}$.
For $\varphi(z):=\Psi(|z|^2)$, $z \in \C^d$,  let $d\mu_\varphi(z)=e^{-\varphi(z)}dm_d(z)$, where $dm_d(z)$ 
denotes the Lebesgue measure on $\C^d$.
Given a separable Hilbert space ${\mathcal H}$, we denote by  $L^2_\varphi({\mathcal H})$ the space of measurable $\mathcal H$-valued  
functions that are square integrable with respect to $d\mu_\varphi$.
We define the vector-valued Fock space ${\mathcal F}_\varphi^2({\mathcal H})$ as the subspace of $L^2_\varphi({\mathcal H})$ consisting of holomorphic functions, i.e.
$${\mathcal F}_\varphi^2({\mathcal H}) = \{ f: \C^d \to {\mathcal H}\ \text{holomorphic}:\ \Vert f \Vert_\varphi^2=
\int_{\C^d} \Vert f(z)\Vert^2 d\mu_\varphi(z) < \infty\}\,.$$
The point evaluations are bounded linear maps from ${\mathcal F}_\varphi^2({\mathcal H})$ to ${\mathcal H}$: more precisely, for any $f\in {\mathcal F}_\varphi^2({\mathcal H})$ we have 
\begin{equation}\label{pointev}
\|f(z)\|\le c(z) \|f\|_\varphi, \quad z\in\C^d,
\end{equation}
where $c(z)\asymp e^{\Psi(|z|^2)/2} \Phi'(|z|^2)^{1/2}(\Psi'(|z|^2))^{(d-1)/2}$.  For $\dim \mathcal{H}=1$, this estimate was proved
in Lemma 8.2 in \cite{SY}, and the passage to the vector-valued case is straightforward via bounded linear functionals. 
It follows that ${\mathcal F}_\varphi^2(\mathcal{H})$ is  a closed subspace of $L^2_\varphi({\mathcal H})$ and hence 
the orthogonal projection from $L^2_\varphi({\mathcal H})$ onto 
${\mathcal F}_\varphi^2({\mathcal H})$ is given by
\begin{equation}\label{proj}
(P_\varphi f)(z)=\int_{\C^d} K_\varphi(z,w)\,f(w)\,d\mu_\varphi(w)\,,\qquad z \in \C^d\,,
\end{equation}
where $\C^d \times \C^d \ni (z,w) \mapsto K_\varphi(z,w)$ denotes the reproducing kernel of the scalar Fock space 
${\mathcal F}_\varphi^2(\C)$.
Again, the last formula is easily deduced from the reproducing formula of the scalar Fock space ${\mathcal F}_\varphi^2(\C)$ applied to 
$z \mapsto \langle P_\varphi f(z),h\rangle$, where $h \in {\mathcal H}$ is arbitrary. 

We are now ready to define vectorial Hankel operators. 
In what follows, $\mathcal{L}({\mathcal H})$ will stand for the space of bounded linear operators on $\mathcal{H}$ and
$ {\mathcal K}({\mathcal H})$ will stand for the space of compact linear operators on ${\mathcal H}$. We denote by ${\mathcal T}_\varphi({\mathcal L}({\mathcal H}))$ the space of 
holomorphic  operator-valued functions $T: \C^d \to {\mathcal L}({\mathcal H})$ 
that satisfy
$$K_\varphi(\cdot,z)\Vert T(\cdot)\Vert_{{\mathcal L}({\mathcal H})} \in L_\varphi^2(\C^d)\ \text{for all}\,z\in\C^d.$$
For $T \in {\mathcal T}_\varphi({\mathcal L}({\mathcal H}))$ we define the {\it big Hankel operator} $H_{T^\ast}$ with symbol $T^\ast$ by
$$H_{T^\ast}f(z):=(I-P_\varphi)(T(\cdot)^\ast f(\cdot))(z)= \int_{\C^d} [T(z)^\ast-T(w)^\ast]f(w) \cdot K_\varphi(z,w)\,d\mu_\varphi(w)$$
for all $f\in {\mathcal F}_\varphi^2({\mathcal H})$. 

In the scalar case, the boundedness/compactness 
of such operators was shown  to be equivalent to their symbols belonging to the  Bloch space/little Bloch space (see \cite{SY}). Moreover, the Schatten
class membership is equivalent to the symbol belonging to analytic Besov spaces. 

We recall that the scalar Bloch space in several complex variables was first introduced by Timoney \cite{T1, T2} for bounded symmetric
domains.  
The scalar Bloch space ${\mathcal B}$, corresponding to the weight $e^{-\Psi(|z|^2)}$ on $\C^d$, was considered by Seip and Youssfi \cite{SY} 
and is defined as the space of holomorphic functions $f: \C^d \to \C$ with
\begin{equation}\label{scalar-bloch}
\Vert f \Vert_{{\mathcal B}}=\sup_{z \in \C^d}\, 
\Big\{ \sup_{\xi \in \C^d,\, \xi \neq 0} \frac{|\langle \nabla f(z),\,\bar{\xi}\rangle|}{\beta(z,\xi)}\Big\} < \infty,
\end{equation}
where  $\beta(z,\xi)$ denotes the Bergman metric 
$$\beta(z,\xi)=\sqrt{\langle B(z)\xi,\,\xi\rangle}\,,\qquad z, \xi \in \C^d\,,$$
and $B(z)$ is the $d \times d$-matrix with entries
$$\Big[ \frac{\partial^2}{\partial \bar{z}_j\,\partial z_k}\,\log\,K(z,z)\Big]_{jk}\,,\qquad 1 \le j,k \le d\,.$$
$B(z)$ is positive-definite and it is usually referred to as the Bergman matrix. 
A standard argument (see e.g. \cite{Z}) shows that
\begin{equation}\label{bloch}
\Vert f \Vert_{{\mathcal B}}= \sup_{z \in \C^d}\sqrt{ \langle\, \overline{B^{-1}(z)}\,\nabla f(z),\,\nabla f(z)\rangle}\,.
\end{equation}

\noindent We shall now define an operator-valued version of the Bloch space ${\mathcal B}$, for which
we provide several adequate equivalent norms.
One of these is an analogue of \eqref{scalar-bloch} (see Section 2), the second one is expressed in terms of  mean oscillation (see \eqref{norm-bmo}), 
and it turns out that our Bloch space coincides with an operator-version of BMOA.
We now define a third norm
which is more relevant for our approach in studying the Hankel operator. 
Inspired by (\ref{bloch}) we introduce
\begin{equation}\label{qt}
Q_T(z):=\sum_{1\le i,j \le d} \overline{B^{-1}(z)}_{ij} D_j T(z)  \Big( D_i T(z) \Big)^* , \quad z\in \C^d,
\end{equation}
where $B^{-1}(z)_{ij}$ denotes the $(ij)$-th entry of the hermitian matrix $B^{-1}(z)$ and $D_i T(z)^*$ is the adjoint of the operator
$D_i T(z)=\frac{\partial T}{\partial z_i}(z)$.

The {\it operator-valued Bloch space} $\mathcal{B}({\mathcal{L}(\mathcal H}))$ is the space of holomorphic functions $T:\C^d\rightarrow \mathcal{L}({\mathcal H})$ with 
\begin{equation}\label{vBloch}
\|T\|_{\mathcal{B}(\mathcal{L}({\mathcal H}))}=\|T(0)\|_{{\mathcal{L}(\mathcal H})}+\sup_{z\in\C^d} \|Q_T(z)\|_{{\mathcal{L}(\mathcal H})}^{1/2}<\infty.
\end{equation}
Notice that for ${\mathcal H}=\C$, in view of (\ref{bloch}), we recover the scalar Bloch space. 
The operator $Q_T(z)$ can be expressed in terms of the radial derivative, the tangential derivatives of $T$, as well as the
eigenvalues of $B(z)$ (see Section 2). Taking this into account, we show that $\mathcal{B}({\mathcal{L}(\mathcal H}))$ can be characterized as
the space of holomorphic functions $T:\C^d\rightarrow {\mathcal{L}(\mathcal H})$ with the property that there exist $c_1,c_2>0$ such that
\begin{gather*}
\|RT(z)\|_{{\mathcal L}({\mathcal H})}\le c_1 |z|\sqrt{\Phi'(|z|^2)}\text{ and }\|T_{ij}(T)(z)\|_{{\mathcal L}({\mathcal H})}\le c_2 |z|\sqrt{\Psi'(|z|^2)}, \quad 1\le i,j\le d,
\end{gather*}
where $RT$ denotes the radial derivative of $T$, and $T_{ij}(T)$ denote the tangential derivatives of $T$, i.e.
\begin{equation}\label{radial-tang}
RT(z)=\sum_{k=1}^d z_k  \frac{\partial T}{\partial z_k},\quad\quad 
T_{ij}(T)=\overline{z}_i\, \frac{\partial T}{\partial z_j} - \overline{z}_j\, \frac{\partial T}{\partial z_i},\quad 1\le i,j\le d.
\end{equation}

For a continuous function $f:\C^d\rightarrow \mathcal{L} \left(\mathcal{H}\right)$ such that $\left\|f(.)\right\|_{\mathcal{L} \left(\mathcal{H}\right)}\left|k_z\right| ^2$ is in $L^1\left(d\mu_{\varphi}\right)$ for all $z$, one defines its  Berezin transform analogously to the scalar case by 
$$ \tilde f (z)=\int_{\C^d}f(w)\left|k_z(w)\right| ^2 d\mu_{\varphi}(w),\ z\in\C^d.$$
 For a continuous function  $T:\C^d\rightarrow \mathcal{L} \left(\mathcal{H}\right)$ we define
\begin{equation}\label{def mo}
	MO^2 T^*(z):=\widetilde{TT^*}(z)-\tilde T(z) \tilde T^*(z), \ z\in\C^d,
\end{equation}
provided  $\left\|T(.)\right\|_{\mathcal{L} \left(\mathcal{H}\right)}\left|k_z\right|$ is in $L^2_\varphi(\C)$ for all $z\in\C^d$.
We say that $T$ has {\it bounded mean oscillation} if 
$\sup_{z\in\C^d}\left\|MO^2 T^*(z)\right\|_{\mathcal{L} \left(\mathcal{H}\right)}<\infty$ and we introduce the norm
\begin{equation}\label{norm-bmo}
	\left\|T\right\|_{BMO(\mathcal{L} \left(\mathcal{H})\right)}:=\sup_{z\in\C^d}\left\|MO^2 T^*(z)\right\|^{1/2}_{\mathcal{L} \left(\mathcal{H}\right)}+\left\|T(0)\right\|_{\mathcal{L} \left(\mathcal{H}\right)}.
\end{equation}
For the connection between Hankel operators and bounded mean oscillation see also \cite{bekolle}.


Throughout this paper, for two functions $E_1,E_2$,  the notation $E_1\lesssim E_2$ means that
there is a constant $k > 0$ independent of the argument such that $E_1\le k E_2$. If both $E_1\lesssim E_2$ and $E_2\lesssim E_1$ hold,
then we write $E_1\asymp E_2$.

The next theorem characterizes the boundedness of the Hankel operator $H_{T^*}$. 

\begin{theorem}\label{boundedness}
Given a holomorphic function $T:\C^d\rightarrow {\mathcal L}({\mathcal H})$, the following are equivalent:
\begin{enumerate}
\item[$(a)$]  $T \in {\mathcal T}_\varphi({\mathcal L}({\mathcal H}))$ and the Hankel operator $H_{T^\ast}$ is bounded 
 from ${\mathcal F}_\varphi^2({\mathcal H})$ to ${ L}_\varphi^2({\mathcal H}) ;$
\item [$(b)$]$ T \in {\mathcal B}\left({\mathcal L}({\mathcal H})\right)$;
\item [$(c)$]  $\sup_{z\in\C^d}\left\|MO^2 T^*(z)\right\|^{1/2}_{\mathcal{L} \left(\mathcal{H}\right)}<\infty$.
\end{enumerate}
Moreover,
$$ \left\|T\right\|_{ \mathcal{B}\left(\mathcal{L} \left(\mathcal{H}\right)\right)}\asymp
\Bigl(\left\|H_{T^*}\right\|+\left\|T(0)\right\|_{\mathcal{L} \left(\mathcal{H}\right)}\Bigr)
\asymp \left\|T\right\|_{BMO(\mathcal{L} \left(\mathcal{H})\right)}.
$$
\end{theorem}

 Inspired by the scalar case \cite{SY}, we present  two alternative proofs of the implication  $(b)\Rightarrow(a)$ above:
one of them relies on the Schur test combined with the reproducing kernel estimates provided in \cite{SY}, while
the second one is based on H\"ormander estimates for the $\bar\partial$-equation. Due to non-commutativity, the latter proof is not a mere adaptation
of the one from the scalar case, and it
provides an estimate in terms of  the multiplication operator with symbol the operator-valued function $Q_T^{1/2}$, which will be used in an essential way in the  characterizations of  compactness and  Schatten class membership of the Hankel operator.


Subsequently, in Theorem \ref{compactness} (see Section 4) we show that a "little oh" version of condition $(b)$, respectively $(c)$, from Theorem \ref{boundedness}
characterizes  the compactness of $H_{T^*}$.

We recall that, given two separable Hilbert spaces $H_1, H_2$ and $p>0$, a compact linear operator $A:H_1\rightarrow H_2$ belongs to the Schatten class
$\mathcal{S}^p=\mathcal{S}^p(H_1, H_2)$ if the sequence of eigenvalues $\{s_n\}_n$ of $(T^*T)^{1/2}$ satisfies 
$$
\|A\|_{\mathcal{S}^p}:=(\sum_{n} s_n^p)^{1/p}<\infty.
$$ 
The Schatten class membership of $H_{T^*}$ is characterized below.
\begin{theorem}\label{Schatten}
Suppose $T:\C^d\rightarrow \mathcal{K}(\mathcal{H})$ is holomorphic and $p\ge 2$. Then the following are equivalent:
\begin{enumerate}
\item[$(a)$]  
 $T \in {\mathcal T}_\varphi({\mathcal L}({\mathcal H}))$ and the Hankel operator $H_{T^*}$ belongs to the Schatten class $\mathcal{S}^p({\mathcal F}_\varphi^2({\mathcal H}), L_\varphi^2({\mathcal H}))$;
\item [$(b)$]
$Q_T^{1/2}:\C^d\rightarrow {\mathcal{S}^p}(\mathcal{H})$ is measurable and 
\begin{equation}\label{besov}
\int_{\C^d} \|Q_T(z)^{1/2}\|^p_{  \mathcal{S}^p(\mathcal{H})} K(z,z)\, d\mu_\varphi(z)<\infty;
\end{equation}
\item [$(c)$]  
$(MO^2 T^*)^{1/2}:\C^d\rightarrow {\mathcal{S}^p}(\mathcal{H})$ is measurable and 
\begin{equation}\label{SCH-BMO}
\int_{\C^d}\|\left(MO^2 T^*(z)\right)^{1/2}\|^{p}_{\mathcal S^p(\mathcal H)}  K(z,z)\, d\mu_\varphi(z)<\infty.
\end{equation}
\end{enumerate}
Moreover, we have  equivalence between the following quantities
$$ \|H_{T^*}\|_{\mathcal{S}^p}\asymp \left\|(Q_T)^{1/2}\right\|_{L^p(\C^d,\mathcal{S}^p({\mathcal H}), d\lambda_\varphi)}
\asymp \left\|(MO^2 T^*)^{1/2}\right\|_{L^p(\C^d,\mathcal{S}^p({\mathcal H}), d\lambda_\varphi)},$$
where $d\lambda_\varphi(z):=K(z,z)\, d\mu_\varphi(z)$.
\end{theorem}
Similar considerations to the ones in Section 9 in \cite{SY}, show that there 
is no nontrivial holomorphic function $T:\C^d\rightarrow \mathcal{L}(\mathcal{H})$ such that condition \eqref{besov} holds for $p=2$,  and therefore
there are no nontrivial Hilbert-Schmidt Hankel operators with anti-holomorphic symbols on $\mathcal{F}^2_\varphi(\mathcal{H})$.

Here, it is worthwhile mentioning the following specificity of  the vector-valued setting  in our approach to prove
the necessity of the conditions on the symbol $T$ for compactness, respectively Schatten class membership. At a first glance, the 
test functions that seem natural to consider are of the type   $k_z e$, where $k_z$ is the normalized reproducing kernel of ${\mathcal F}_\varphi^2({\C})$ and
$e\in {\mathcal H}$. However, it turns out that we need to consider test functions of the form $k_z e_z$, for an appropriate choice
of the vectors $e_z$, that depends on the operator-valued function $Q_T(z)$. 

Regarding previous studies of big Hankel operators on scalar Fock spaces
we would also like to mention \cite{bauer,bauer1,BY2,ortega,LR1,WCZ}, as for Hankel forms on vector-valued Bergman-type spaces we refer
to \cite{aleman,aleman1}.

The paper is organized as follows. Section 2 is concerned with equivalent definitions and basic properties of the operator-valued Bloch, little Bloch space, as well as some preliminary material. Section 3 is dedicated to the boundedness of $H_{T*}$, while in Section 4 we characterize
the compactness of   $H_{T*}$. Finally, in Section 5 we investigate the Schatten class membership of our Hankel operators.

\section{The operator-valued Bloch space, little Bloch space and BMOA}
We start with some considerations regarding the Bergman matrix.
Recall that the Bergman matrix  $B(z)$ is the $d \times d$-matrix with entries
$$\Big[ \frac{\partial^2}{\partial \bar{z}_j\,\partial z_k}\,\log\,K(z,z)\Big]_{jk}\,,\qquad 1 \le j,k \le d\,.$$
Notice that if $F(|z|^2):=K(z,z)$, then
$$B(z)=\frac{F'}{F}\,I + |z|^2\Big( \frac{F'}{F}\Big)' P_z\,,$$
where $I$ stands for the identity matrix, $P_z$ denotes  the projection of $\C^d$ onto $\text{span}\{z\}$, given by
$$P_zw
=\frac{1}{|z|^2}\,\langle w,\,z\rangle\, z\,\quad z,w\in\C^d .$$
We can rewrite
$$B(z)=\lambda(z)\,P_z + \mu(z)\,(I-P_z)\,,$$
where
\begin{eqnarray*}\label{}
\lambda(z) =  \frac{F'}{F}(|z|^2) + |z|^2\Big( \frac{F'}{F}\Big)'(|z|^2)\ \ \ {\hbox {and }}\ \ 
\mu(z) = \frac{F'}{F}(|z|^2)
\end{eqnarray*}
are the eigenvalues of $B(z)$.
Hence
\begin{equation}\label{inv}
\big(B(z)\big)^{-1}=\frac{1}{\lambda(z)}\,P_z + \frac{1}{\mu(z)}\,(I-P_z)\,.
\end{equation}
 Now Lemma 4.1 from \cite{SY} gives
\begin{eqnarray*}
 \frac{F'}{F}(r)&=&(1+o(1))\Psi'(r), \\
 \Bigl(\frac{F'}{F}\Bigr)'(r)&=&(1+o(1))\Psi''(r)+o(1)\frac{\Psi'(r)}{r}, \ \ \hbox{ as }  r\rightarrow \infty,
\end{eqnarray*}
which implies
\begin{eqnarray}\label{eigenv-comp}
\lambda(z)&\asymp&  \Psi'(|z|^2)+|z|^2 \Psi''(|z|^2)\nonumber=\Phi'(|z|^2)\\
\mu(z)&\asymp&  \Psi'(|z|^2).
\end{eqnarray}


It was shown in  Lemma 7.2 in \cite{SY}  that, instead of working with Bergman balls (i.e. balls corresponding to 
the Bergman distance), one can equivalently work with sets of the form  
\begin{equation}\label{polyball}
D(z,a)= \{w\ : \ |z-P_z w|\le a [\Phi'(|z|^2)]^{-1/2},\ |w-P_z w|\le a [\Psi'(|z|^2)]^{-1/2}  \},
\end{equation}
where $z\in\C^d,\,a>0$.
\begin{lemma}
The sets $D(z,a)$ are unitarily invariant, that is,  if $U: \C^d \to \C^d$ is a unitary map, then 
\begin{equation}\label{il}
U(D(z,a))=D(Uz,a),\quad z \in \C^d, a>0.
\end{equation}
\end{lemma}

\begin{proof} The proof is straightforward and relies on the identity $UP_z=P_{Uz}U$ for any $z \in \C^d$.
\end{proof}


Recall that in \eqref{qt} we introduced the operator
\begin{equation*}
Q_T(z):=\sum_{1\le i,j \le d} \overline{B^{-1}(z)}_{ij} D_j T(z)  \Big( D_i T(z) \Big)^* , \quad z\in \C^d,
\end{equation*}
Depending on the context, we shall use alternative expressions for $Q_T(z)$.  From (\ref{inv}) we have 
$$B^{-1}(z)_{ij}=\left(\frac{1}{\lambda(z)}-\frac{1}{\mu(z)}\right) \frac{z_i \,\overline{z}_j}{|z|^2} + \frac{1}{\mu(z)}\delta_{ij}\,.$$
Substituting this in the expression of  $Q_T(z)$ we obtain for $z\neq 0$
\begin{eqnarray}\label{b2}
Q_T(z)
&=& \frac{1}{|z|^2}\,\Big(\frac{1}{\lambda(z)} - \frac{1}{\mu(z)}\Big) \, RT(z) \Big( RT(z)\Big)^\ast + \frac{1}{4\mu(z)}\, \Delta(T(z) T(z)^\ast ) \nonumber\\
&=& \frac{1}{\lambda(z) |z|^2} \, RT(z)\Big( RT(z)\Big)^\ast  + \frac{1}{\mu(z) |z|^2} \,\sum_{1\le i<j\le d} T_{ij}(T)(z)\Big(T_{ij}(T)(z)\Big)^\ast \,,
\end{eqnarray}
where $RT$ and  $T_{ij}(T)$ were defined in \eqref{radial-tang}.
In particular, this shows that $Q_T(z)$ is a positive operator.

\noindent Another expression of $Q_T$ which will be useful in several of the subsequent proofs is the following. 
If $c_{kj}(z)$ stands for the $kj$ entry of the (hermitian) matrix ${B(z)^{-1/2}}$, where $1\le k,j\le d$, set
\begin{equation}\label{cj}
C_j (z):= \sum_{k=1}^d c_{kj}(z)\, D_k T(z),\quad z\in \C^d.
\end{equation}
Obviously $C_j (z)\in \mathcal{L}(\mathcal{H})$ and we have 
\begin{eqnarray}\label{QT-Cj}
\sum_{j=1}^d C_j(z) \Big(C_j(z)\Big)^\ast  
&=& \sum_{k,l=1}^d \Big( \sum_{j=1}^d c_{lj}(z)\, c_{jk}(z)\Big)  D_l T(z) \Big(D_k T(z)\Big)^\ast\nonumber \\
&=& \sum_{1\le k,l \le d} \Big(B^{-1}(z) \Big)_{lk} D_l T(z) \Big(D_k T(z)\Big)^\ast=Q_T(z)\,.
\end{eqnarray}
A straightforward calculation shows that $Q_T$ satisfies
\begin{equation}\label{itz}
	\left\|Q_{T+S}(z)\right\|^{1/2}_{{\mathcal L}({\mathcal H})}\leq\left\|Q_{T}(z)\right\|^{1/2}_{{\mathcal L}({\mathcal H})}+\left\|Q_{S}(z)\right\|^{1/2}_{{\mathcal L}({\mathcal H})},\quad T,S\in \mathcal{L}(\mathcal{H}),\  z\in\C^d,
\end{equation}
which implies that \eqref{vBloch} defines a norm on $\mathcal{B}\left(\mathcal{L}\left(\mathcal{H}\right)\right)$. 
Moreover, the completeness of $\mathcal{B}\left(\mathcal{L}\left(\mathcal{H}\right)\right)$ follows by a standard argument similar to the one in \cite{Z}.

In the next proposition we provide an equivalent norm on  $\mathcal{B}({\mathcal{L}(\mathcal H}))$, which is an analogue of \eqref{scalar-bloch},
and we prove the vectorial version of a standard estimate for Bloch functions in terms of the Bergman distance.  The Bergman distance is 
defined by
$$
d_\Psi (z,w):= \inf_\gamma \int_0^1 \beta(\gamma(t),\gamma'(t))\, dt, \quad z,w\in\C^d,
$$
where the infimum is taken over all piecewise $C^1-$smooth curves $\gamma:[0,1]\rightarrow \C^d$ such that $\gamma(0)=w$ and $\gamma(1)=z$.

\begin{proposition}\label{schur}
(a) We have  
\begin{equation} 
\Vert T \Vert_{\mathcal{B}({\mathcal{L}(\mathcal H}))}\asymp \|T(0)\|_{{\mathcal{L}(\mathcal H})} + \sup_{z \in \C^d}\, 
\Big\{ \sup_{\xi \in \C^d,\, \xi \neq 0} \frac{ \| \sum_{k=1}^d \xi_k D_k T(z)\|_{\mathcal{L}(\mathcal{H})}}{\beta(z,\xi)}\Big\},
\end{equation}
for all holomorphic functions $T:\C^d\rightarrow \mathcal{L}({\mathcal H})$, where the involved constants depend only on $d$.

(b) For any $T\in \mathcal{B}({\mathcal{L}(\mathcal H}))$ we have
\begin{equation}\label{schur-estimate}
\|T(z)-T(w)\|_{ \mathcal{L} (\mathcal{ H})}
\lesssim 
\Vert T \Vert_{\mathcal{B} ( \mathcal{L} (\mathcal{H}) )} d_\Psi (z,w),\quad z,w \in \C^d,
\end{equation}
where $d_\Psi$ denotes the Bergman distance induced by the Bergman metric $\beta$.
\end{proposition}

\begin{proof}
$(a)$ Using the fact that $\beta(z,\xi)=\sqrt{\langle B(z)\xi,\,\xi\rangle}$ and substituting $\eta:=B(z)^{1/2} \xi$ we may write
\begin{eqnarray}\label{chain}
E(z)&:=&\sup_{\xi \in \C^d,\, \xi \neq 0} \frac{ \| \sum_{k=1}^d \xi_k D_k T(z)\|_{{\mathcal L}({\mathcal H})}}{\beta(z,\xi)}=\sup_{\eta\in \C^d,\, \eta \neq 0}
  \frac{ \| \sum_{k=1}^d (B(z)^{-1/2} \eta)_k D_kT(z)\|_{{\mathcal L}({\mathcal H})}}{\|\eta\|}\nonumber\\
&=&  \sup_{w\in \C^d,\, \|w\|=1}
  \| \sum_{k=1}^d (B(z)^{-1/2} w)_k D_k T(z)\|_{{\mathcal L}({\mathcal H})} 
= \sup_{w\in \C^d,\, \|w\|=1}
  \| \sum_{j=1}^d w_j  C_j (z)\|_{{\mathcal L}({\mathcal H})},\nonumber\\
\end{eqnarray}
where, in the last two steps above, we used the notation from \eqref{cj}.  Particularizing $w$ in the last expression above to the vectors from the canonical basis of $\C^d$, we  obtain
\begin{equation}
E(z)\ge \|C_j (z)\|_{{\mathcal L}({\mathcal H})}=\|C_j(z)C_j(z)^*\|_{{\mathcal L}({\mathcal H})}^{1/2}\,\quad z\in\C^d, 1\le j\le d.
\end{equation} 
Using the above together with \eqref{QT-Cj} we deduce
\begin{equation}\label{prev}
d\cdot E(z)^2 \ge \sum_{j=1}^d \|C_j(z)C_j(z)^*\|_{{\mathcal L}({\mathcal H})}\ge \| \sum_{j=1}^d C_j(z)C_j(z)^*\|_{{\mathcal L}({\mathcal H})}=\|Q_T(z)\|_{{\mathcal L}({\mathcal H})}.
\end{equation}
On the other hand, relation \eqref{chain} and the Cauchy-Schwarz inequality immediately  give
\begin{eqnarray*}
E(z)^2&\le&   \sup_{w\in \C^d,\, \|w\|=1}( \sum_{j=1}^d |w_j|\|C_j(z)\|_{{\mathcal L}({\mathcal H})})^2\le ( \sum_{j=1}^d \|C_j(z)\|_{{\mathcal L}({\mathcal H})}^2)\\
&\le& d \cdot \| \sum_{j=1}^d C_j(z)C_j(z)^*\|_{{\mathcal L}({\mathcal H})}=d \cdot \|Q_T(z)\|_{{\mathcal L}({\mathcal H})},
\end{eqnarray*}
where the last inequality above follows by positivity.
Together with \eqref{prev} this implies
$$
E(z)\asymp  \|Q_T(z)\|_{{\mathcal L}({\mathcal H})}^{1/2},
$$
and $(a)$ now follows by taking the supremum over $z\in\C^d$ in the above relation.

In order to prove $(b)$, let $z,w\in \C^d$ and consider a piecewise $C^1$ curve $\gamma:[0,1]\rightarrow \C^d\  (\gamma=(\gamma_1,...,\gamma_d))$ such that $\gamma(0)=w$ and $\gamma(1)=z$.
Then, in view of $(a)$, we get
\begin{eqnarray*}
\|T(z)-T(w)\|_{{\mathcal L}({\mathcal H})}
&\le&  \int_0^1  \| \sum_{j=1}^d \gamma_j'(t) \,D_jT(\gamma(t))\|_{{\mathcal L}({\mathcal H})}\, dt  
\\
&\le& \sqrt{d} \sup_{\xi\in\C^d}\Vert Q_T(\xi) \Vert^{1/2}_{{\mathcal L}({\mathcal H})}
\int_0^1  \beta(\gamma(t), \gamma'(t)) \, dt.
\end{eqnarray*}
Taking now the supremum over $\gamma$ above leads us to  (b).
\end{proof}

The next lemma, which is a direct consequence of the reproducing kernel estimates proven in \cite{SY},  shows that the operator-valued Bloch space is contained in ${\mathcal T}_\varphi({\mathcal L}({\mathcal H}))$.
\begin{lemma}\label{integral}
If $T \in {\mathcal B}({\mathcal L}({\mathcal H}))$, then for any $z \in {\mathbb C}^d$ we have 
$\Vert T(\cdot) \Vert_{{\mathcal L}({\mathcal H})} \cdot K(\cdot,z) \in L_\varphi^2({\mathbb C})$.
\end{lemma}

\begin{proof}
Since $\Psi \in {\mathcal S}$ and satisfies \eqref{psi} we have
$$[\Psi'(x)]^{-\eta} \Psi''(x) \lesssim \Psi'(x)\,,\qquad x \ge 0\,,$$
which implies
\begin{equation}\label{A}
\Psi'(x) \lesssim [1+\Psi(x)]^{\frac{1}{1-\eta}}\,,\qquad x \ge 0\,.
\end{equation}
Hence
\begin{equation}\label{B}
\Phi'(x) =\Psi'(x) + x\Psi''(x) \lesssim (1+x)[1+\Psi(x)]^3
\end{equation}
since $\eta <\frac{1}{2}$. From the fact that $T \in {\mathcal B}({\mathcal L}({\mathcal H}))$ together with \eqref{eigenv-comp} we obtain the estimate
$$\Vert R_T(z) \Vert_{{\mathcal L}({\mathcal H})}^2 \lesssim \lambda(z)|z|^2 \asymp \Phi'\left(|z|^2\right) |z|^2\,,\qquad z \in {\mathbb C}^d.$$
Combining this with \eqref{B} and taking into account the fact that $\Phi'$ is increasing, we deduce
\begin{eqnarray*}
\Vert T(z)-T(0) \Vert_{{\mathcal L}({\mathcal H})} &=& \Bigl\Vert \int_0^1 \frac{1}{t}\,R_T(tz)\,dt \Bigr\Vert_{{\mathcal L}({\mathcal H})} 
\le |z| \sqrt{(1+|z|^2)(1+\Psi(|z|^2))^3} \,,\qquad z \in {\mathbb C}^d.
\end{eqnarray*}
Since $\Psi$ grows at least like a linear function, the above estimate yields
\begin{eqnarray*}
I=I(z) &:=& \int_{{\mathbb C}^d} \Vert T(w)\Vert_{{\mathcal L}({\mathcal H})}^2 |K(w,z)|^2 e^{-\Psi(|w|^2)}\,dm_d(w) \\
&\lesssim& \int_{{\mathbb C}^d}  |K(w,z)|^2 e^{-(1-\varepsilon)\Psi(|w|^2)}\,dm_d(w)
\end{eqnarray*}
for any $\varepsilon \in (0,1/2)$. Recall that $K(w,z)=F(\langle z,w\rangle)$. 
Then, by unitary invariance, we may assume without loss of generality that $z=(x,0,\dots, 0)$ with $x>0$. If $d>1$ we write 
$w=(w_1,\xi)$ with $\xi \in {\mathbb C}^{d-1}$ and $w_1=re^{i\theta}$, and use polar coordinates to get
\begin{equation}\label{ast0}
I \lesssim \int_0^\infty \int_{-\pi}^{\pi} |F(xre^{i\theta})|^2 
\Big( \int_{{\mathbb C}^{d-1}} e^{-(1-\varepsilon)\Psi(r^2+|\xi|^2)}\,dm_{d-1}(\xi)\Big)\,r\,d\theta dr\,.
\end{equation}
Using again the monotonicity of $\Psi$ we deduce
\begin{eqnarray}\label{astast}
 \int_{{\mathbb C}^{d-1}} e^{-(1-\varepsilon)\Psi(r^2+|\xi|^2)}\,dm_{d-1}(\xi) &\lesssim& 
 e^{-(1-2\varepsilon)\Psi(r^2)} \int_{{\mathbb C}^{d-1}} e^{-\varepsilon\Psi(|\xi|^2)}\,dm_{d-1}(\xi) \nonumber \\
  &\lesssim& e^{-(1-2\varepsilon)\Psi(r^2)}\,.
\end{eqnarray}

The estimates of the reproducing kernel (see Lemma 3.1 in \cite{SY}) together with \eqref{A}-\eqref{B} give
$$\int_{-\pi}^{\pi} |F(xre^{i\theta})|^2 d\theta \lesssim (1+xr)^{3/2}[1+\Psi(xr)]^N e^{2\Psi(xr)}\,,$$
where $N=N(d)>0$ and the constants involved  depend on $x$, but not on $r$. Taking into account the above 
relation and \eqref{astast}, we now return to \eqref{ast0} to deduce
\begin{eqnarray*}
I 
\lesssim \int_0^\infty e^{-(1-2\varepsilon)\Psi(r^2)+(2+\varepsilon)\Psi(xr)}\,dr\,.
\end{eqnarray*}
To see that the last integral is finite,  put $Q(r)=(1-2\varepsilon)\,\Psi(r^2)-(2+\varepsilon)\Psi(xr)$ and 
notice that for $2(1-2\varepsilon)r- (2+\varepsilon)x \ge 1$ we have
$$Q'(r) =2(1-2\varepsilon)r \Psi'(r^2) -(2+\varepsilon)x\Psi'(xr)  \ge \min_{t \ge 0} \{ \Psi'(t)\}=: \delta >0\,,$$
and hence $e^{-Q(r)} \lesssim e^{-\delta r}$, which proves the claim.   
\end{proof}

\medskip
 Let $\mathcal{M} $ be a closed subspace of $\mathcal{L}\left(\mathcal{H}\right)$. 
The {\it little Bloch space} $\mathcal{B}_0(\mathcal{M})$  is the space of  
holomorphic functions $T:\C^d\rightarrow \mathcal{M}$
such that 
\begin{equation}\label{vBloch0}
\lim_{\left|z\right|\to +\infty}\left\|Q_T(z)\right\|_{\mathcal{L}\left(\mathcal{H}\right)}=0.
\end{equation}

Let us now show that the density of polynomials in the scalar little Bloch space extends to the operator-valued case. The proof of this
fact is standard and it is based on approximation by convolutions with Fej\'er kernels (see \cite{grafakos}). We include it for the sake of completeness.
\begin{theorem}\label{density pol}
 Let $\mathcal{M} $ be a closed subspace of $\mathcal{L}\left(\mathcal{H}\right)$. Then the holomorphic polynomials with coefficients in $\mathcal{M} $ are dense in $\mathcal{B}_0\left(\mathcal{M} \right)$.
\end{theorem}

\begin{proof}
Assume $T\in\mathcal{B}_0\left(\mathcal{M}\right)$.
 For $\left(\theta_1,\cdots,\theta_d \right)\in \mathbb{R}^d$, we consider the unitary linear transformation in $\mathbb{C}^d$ defined by
$R_{\theta}(z):=\left(e^{i\theta_1}z_1,\cdots,e^{i\theta_d}z_d\right)  $, for all $z=\left(z_1,\cdots,z_d\right)\in \C^d$. 
The torus $\T^d=\left\{\left(e^{i\theta_1},\cdots,e^{i\theta_d}\right),\ \left(\theta_1,\cdots,\theta_d\right)\in\left[-\pi,\pi\right]^d\right\}$ is equipped with the Haar measure $d\theta$, and, for any nonnegative integer $N$,  the Fej\'er kernel $F_N$ is given by
\begin{equation}\label{Fejer}
F_N\left(e^{i\theta_1},\cdots,e^{i\theta_d}\right):=\sum_{\left|m_j\right|\leq N,\\
m_j\in \Z}\left(1-\frac{\left|m_1\right|}{N+1}\right)\cdots \left(1-\frac{\left|m_d\right|}{N+1}\right)e^{i m\cdot \theta},
\end{equation}
where $m\cdot \theta=m_1\theta_1+\cdots+m_d\theta_d $. The convolution 
\begin{equation}\label{convolution}
T_N(z)=\int_{\T^d}T\left( R_{-\theta}z\right)F_N\left(e^{i\theta_1},\cdots,e^{i\theta_d}\right)d\theta,\ z\in\C^d,
\end{equation}
is then a holomorphic polynomial with coefficients in $\mathcal{M}$, which obviously belongs to $\mathcal{B}_0(\mathcal{M})$, and we have
$$  T_N(z)-T(z)=\int_{\T^d}(T \circ R_{-\theta}-T)(z)\cdot F_N\left(e^{i\theta_1},\cdots,e^{i\theta_d}\right)d\theta,\quad z\in \C^d. $$
We claim that
\begin{equation*}
\lim_{N\rightarrow\infty}\left\|T_N-T\right\|_{\mathcal{B}(\mathcal{L}(\mathcal{H}))}=\lim_{N\rightarrow\infty}\sup_{z\in\C^d} \|Q_{T_N-T}(z)\|_{{\mathcal L}({\mathcal H})} ^{1/2}=0.
\end{equation*}
For fixed $z\in\C^d$, $N_z(T)= \left\|Q_{T}(z)\right\|_{{\mathcal L}({\mathcal H})} ^{1/2}$ defines a semi-norm on $\mathcal{L}(\mathcal{H})$ by (\ref{itz}). Thus
\begin{equation}\label{fejer11}
\left\|Q_{T_N-T}(z)\right\|_{{\mathcal L}({\mathcal H})} ^{1/2}\leq\int_{\T^d}\left\| Q_{(T \circ R_{-\theta}-T)}(z)\right\|_{{\mathcal L}({\mathcal H})} ^{1/2}
\cdot F_N\left(e^{i\theta_1},\cdots,e^{i\theta_d}\right)d\theta.
\end{equation}
Then
\begin{eqnarray}\label{fejer000}
\|Q_{T_N-T}(z)\|_{{\mathcal L}({\mathcal H})} ^{1/2}\le  \int_{ V_\delta }+ \int_{\T^d\setminus V_\delta }\| Q_{(T \circ R_{-\theta}-T)}(z)\|_{{\mathcal L}({\mathcal H})} ^{1/2}
\cdot F_N\left(e^{i\theta_1},\cdots,e^{i\theta_d}\right)d\theta, 
\end{eqnarray}
where $V_\delta$  $(\delta>0)$ denotes the neighborhood of $0$ given by
$$V_\delta:=\left\{\left(\theta_1,\cdots,\theta_d\right)\ :\ \ \left|\theta_j\right|\leq \delta,\, 1\le j\le d\right\}. $$
Now let $\varepsilon>0$.
By the properties of $F_N$, there exists $N_0\in\N$ such that 
\begin{equation}\label{fejer01}
\int_{\T^d\setminus V_\delta }F_N\left(e^{i\theta_1},\cdots,e^{i\theta_d}\right)d\theta\leq\varepsilon, \quad N>N_0.
\end{equation}
Since $T \in\mathcal{B}_0\left(\mathcal{M}\right)$, we may choose $R>0$ such that
$$ \sup_{|z|> R}\left\|Q_{T}(z)\right\|_{{\mathcal L}({\mathcal H})} ^{1/2}< \varepsilon.$$
Then relation (\ref{itz}) together with  the rotation invariance $Q_{T\circ R_{\theta}}(z)=Q_{T}(R_{\theta} z)$  imply
\begin{equation}\label{fejer-1}
\sup_{|z|> R}\left\|Q_{(T\circ R_{\theta}-T)}(z)\right\|_{{\mathcal L}({\mathcal H})} ^{1/2}\leq \sup_{|z|> R}\left\|Q_{T\circ R_{\theta}}(z)\right\|_{{\mathcal L}({\mathcal H})} ^{1/2}+\sup_{|z|> R}\left\|Q_{T}(z)\right\|_{{\mathcal L}({\mathcal H})} ^{1/2}<  2\varepsilon.
\end{equation}
Again, from the rotation invariance of  $Q_{T}$ and the uniform continuity of $D_j T$ on every compact set $\left\{z\in\C^d,\ \left|z\right|\leq R\right\}$, $R>0$, we obtain 
$$
\lim_{\theta\to 0} \sup_{|z|\leq R}\left\|Q_{(T\circ R_{\theta}-T)}(z)\right\|_{{\mathcal L}({\mathcal H})} =0.
$$
Then
we may choose $\delta$ small enough such that  
\begin{equation}\label{fejer-2}
 \sup_{|z|\leq R}\left\|Q_{(T\circ R_{\theta}-T)}(z)\right\|_{{\mathcal L}({\mathcal H})} ^{1/2} < \varepsilon, \quad \theta\in V_\delta.
\end{equation}
Using relations (\ref{fejer01}) and (\ref{fejer-2}) in (\ref{fejer000}) yields
$$
\|T_N-T\|_{\mathcal{B}(\mathcal{L}(\mathcal{H}))}=\sup_{z\in\C^d}\|Q_{T_N-T}(z)\|_{{\mathcal L}({\mathcal H})} ^{1/2}\le 2\varepsilon +2\varepsilon\|T\|_{\mathcal{B}(\mathcal{L}(\mathcal{H}))},
$$
for $N>N_0$, which validates the claim, and, thus, completes the proof.
\end{proof}

\section{Boundedness of Hankel operators}
In this section we prove different characterizations of the boundedness of the big Hankel operator. We shall use the notation  
${\mathcal F}_\varphi^2(\mathcal{L}(\mathcal{H}))$ for the Fock space of holomorphic functions
$f: \C^d \to {\mathcal L}({\mathcal H})$ 
that satisfy
$\Vert f(\cdot)\Vert_{{\mathcal L}({\mathcal H})} \in L_\varphi^2(\C^d)$.

\begin{proof}[Proof of Theorem \ref{boundedness}]
\ \\

{\bf Implication} ${\bf (b)\Rightarrow(a)}$. 
Assume that $T \in {\mathcal B}({\mathcal L}({\mathcal H}))$. Lemma \ref{integral} shows that $T \in {\mathcal T}_\varphi({\mathcal L}({\mathcal H}))$. Let $f$ be a holomorphic polynomial in 
$\C^d$ with coefficients in ${\mathcal H}$ and let $\{e_i\}_{i \ge 1}$ be an orthonormal basis of ${\mathcal H}$. Set
\begin{eqnarray*}
F_i(z):= \langle H_{T^\ast}f(z),\,e_i\rangle 
=(I-P_{\varphi})G_i\,(z) \,,
\end{eqnarray*}
where
$$G_i(z)= \langle T(z)^\ast f(z),\, e_i\rangle. $$
 The form
$$\Omega_i := \sum_{j=1}^d \langle f(z),\, D_jT(z)e_i\rangle d\overline{z}_j$$
is closed, that is, $\overline{\partial}\Omega_i=0$. For $1 \le j \le d$, we denote $\Omega_i^j:=\langle f(z),\, D_jT(z)e_i\rangle$. 
Notice that $F_i$ is the solution of minimal $L^2_\varphi(\C)$-norm of 
$$\overline{\partial} u = \Omega_i\,.$$

By a theorem due to H\"ormander (see \cite{berndtsson-c,SY}) it follows that 
\begin{equation}\label{hormander}
\int_{\C^d} |F_i|^2 \,  d\mu_\varphi \le \int_{\C^d} |\Omega_i|^2_{i\partial\bar\partial \varphi}\, d\mu_\varphi.
\end{equation}
Here $ |\Omega_i|_{i\partial\bar\partial \varphi}$  denotes the norm
of $\Omega_i$ measured in the K\"ahler metric defined by ${i\partial\bar\partial \varphi}$, that is
$$
|\Omega_i|^2_{i\partial\bar\partial \varphi}=\sum_{1\le j,k\le d} A^{jk} \Omega_i^j \overline{\Omega_i^k},
$$
where $(A^{jk}(z))_{1\le j,k\le d}$ is the inverse of the hermitian matrix $A(z)=(A_{jk}(z))_{1\le j,k\le d}:=\Bigl( \frac{\partial^2 \varphi}{\partial z_j\partial \bar z_k}(z)\Bigr)_{1\le j,k\le d}$.

Our next aim is to obtain an appropriate estimate for the  right-hand-side of (\ref{hormander}).
Setting 
$X_i:= (\langle f(z), D_1T(z) e_i\rangle,\, ... \,, \langle f(z), D_dT(z) e_i \rangle )\in\C^d$, we may rewrite the last relation above as
\begin{equation}\label{m1}
|\Omega_i|^2_{i\partial\bar\partial \varphi}=\langle (A^{-1}(z))^t X_i,X_i\rangle_{\C^d}=\langle \overline{A^{-1}(z)}X_i,X_i\rangle_{\C^d}.
\end{equation}
Let us now take a closer look at  $A(z)$. We have
\begin{equation*}
A(z)
=\Psi'(|z|^2)I+ (z_k\bar z_j \Psi''(|z|^2))_{1\le j,k \le d}, \quad z\in \C^d.
\end{equation*}
It follows that
$$
\overline{A(z)}
= \tilde\lambda(z) P_z +\tilde\mu(z)(I-P_z),
$$
where 
$$\tilde\lambda(z)= \Psi'(|z|^2)+|z|^2 \Psi''(|z|^2)\hbox{\ \  and\ \  }\tilde\mu(z)=\Psi'(|z|^2).$$
Relation \eqref{eigenv-comp} shows that $\tilde\lambda(z)\asymp\lambda(z)$ and $\tilde\mu(z)\asymp\mu(z)$.
We clearly have
\begin{equation}
\overline{A(z)}^{\,-1}=\overline{A(z)^{-1}}=\frac{1}{\tilde\lambda(z)} P_z+\frac{1}{\tilde\mu(z)}(I-P_z).
\end{equation}
From relation (\ref{inv}) we now deduce that
the matrices $\overline{A(z)}^{\,-1}$ and $B(z)^{-1}$ have the same eigenvectors and comparable eigenvalues, which implies that
the induced hermitian forms are comparable, i.e.
$$
\langle \overline{A(z)^{-1}}v,v\rangle\asymp \langle {B(z)^{-1}}v,v\rangle,
$$
where the involved constants are independent of $z,v\in\C^d$. Using this in (\ref{m1}) we get
\begin{equation}\label{LH}
|\Omega_i|^2_{i\partial\bar\partial \varphi}\asymp  \langle {B(z)^{-1}}X_i,X_i\rangle=\|B(z)^{-1/2}X_i\|^2.
\end{equation}
As in \eqref{cj}, we denote by $c_{jk}(z)$ the $jk$ entry of the (hermitian) matrix ${B(z)^{-1/2}}$, where $1\le j,k\le d$,
and 
\begin{equation*}
C_j (z)= \sum_{k=1}^d c_{kj}(z)\, D_k T(z).
\end{equation*}
 Writing down the components
of $B(z)^{-1/2}X_i$ with this notation, we deduce
\begin{eqnarray*}
\|B(z)^{-1/2}X_i\|^2&=&\sum_{j=1}^d \, \Bigl| \sum_{k=1}^d c_{jk}(z)\, \langle f(z), D_k T(z) e_i\rangle \Bigr|^2\\
 &=& 
 \sum_{j=1}^d \, \Bigl| \, \langle f(z), C_j(z) e_i\rangle \Bigr|^2= \sum_{j=1}^d |\langle C_j(z)^* f(z), e_i \rangle|^2.
\end{eqnarray*}
We now use the above equality in (\ref{LH}) and return to (\ref{hormander}) to deduce
\begin{equation}
\int_{\C^d} |F_i|^2 \,  d\mu_\varphi \lesssim \int_{\C^d} \sum_{j=1}^d |\langle C_j(z)^* f(z), e_i \rangle|^2  \,  d\mu_\varphi (z).
\end{equation}
Summing up over $i$ and applying the monotone convergence theorem yield
\begin{eqnarray}\label{m4}
\|H_{T^\ast}f\|^2&=& \sum_{i=1}^\infty \int_{\C^d} |F_i|^2\,  d\mu_\varphi \lesssim  \sum_{j=1}^d \int_{\C^d} \sum_{i=1}^\infty|\langle C_j(z)^* f(z), e_i \rangle|^2  \,  d\mu_\varphi(z)\nonumber \\
&=& \sum_{j=1}^d \int_{\C^d} \| C_j(z)^* f(z)\|^2  \,  d\mu_\varphi(z) =\int_{\C^d}  \langle  \sum_{j=1}^d C_j(z)  \Big(C_j(z)\Big)^* f(z), f(z)\rangle  \,  d\mu_\varphi(z)\nonumber\\
&=& \int_{\C^d} \|\Big(\sum_{j=1}^d  C_j(z) \Big( C_j(z)\Big)^*\Big)^{1/2}f(z)\|^2\, d\mu_\varphi(z)\nonumber\\
&\le& (\sup_{z\in\C^d} \|Q_T(z)\|_{{\mathcal L}({\mathcal H})}  )\|f\|^2_\varphi.  
\end{eqnarray}
where, from relation  \eqref{QT-Cj}, we have $Q_T(z)=\sum_{j=1}^d  C_j(z)  (C_j(z))^*$. Hence 
$$
\Vert H_{T^\ast}f \Vert \lesssim (\sup_{z \in \C^d} \Vert Q_T(z)\Vert_{{\mathcal L}({\mathcal H})} ^{1/2}) \Vert f \Vert_\varphi \\
\le \Vert T\Vert_{{ \mathcal{B}\left(\mathcal{L} \left(\mathcal{H}\right)\right)}} \Vert f \Vert_\varphi\,,
$$
and $(b)\Rightarrow(a)$ is proven.

{\bf Implication} ${\bf (a)\Rightarrow(c)}$. Suppose $H_{T^\ast}$ is bounded. For $w \in \C^d$ and $e \in {\mathcal H}$ with $\Vert e \Vert=1$, notice 
that by the reproducing formula, we have
$$H_{T^\ast}(K_w \,e)\,(z)=( T(z)^\ast - T(w)^\ast) e \cdot K_w(z)\,,$$
where $K_w(z)=K(z,w)$ is the reproducing kernel of ${\mathcal F}_\varphi^2(\C)$. 
On the other hand, since $T\in {\mathcal T}_\varphi({\mathcal L}({\mathcal H}))$, 
the reproducing formula in
${\mathcal F}_\varphi^2(\mathcal{L}(\mathcal{H}))$ (which follows from its scalar version via bounded linear functionals)  yields
$\tilde T=T$, $ \tilde T^*=T^*$ and  
\begin{gather}\label{mo and HTkz}
 MO^2 T^*(z)=\int_{\C^d}\left(T(z)-T(w)\right)\left(T(z)-T(w)\right)^*\left|k_z(w)\right|^2d\mu_{\varphi}(w).\nonumber
\end{gather}
In particular this shows that $MO^2 T^*(z)$ is a positive operator.
Combining the last two equalities above we deduce
\begin{equation}\label{m-osc}
\left\langle MO^2 T^*(z)e,e\right\rangle_{\mathcal{H}} =
\int_{\C^d} \Vert ( T(z)^\ast - T(w)^\ast) e \Vert^2 |k_z(w)|^2\,d\mu_\varphi(z)=
\left\|H_{T^*}(k_ze)\right\|^{2}_{\mathcal F_\varphi^2},\  e\in\mathcal H.
\end{equation}
Taking the supremum over unit vectors $e\in\mathcal{H}$ in the last relation above, we obtain
$$
\left\|MO^2 T^*(z)\right\|^{1/2}_{\mathcal{L} \left(\mathcal{H}\right)}= \sup_{\left\|e\right\|=1}\left\|H_{T^*}(k_ze)\right\|_{\mathcal F_\varphi^2}\le \| H_{T^*}\|, \quad  z\in\C^d,
$$
and $(c)$ follows.

{\bf Implication} ${\bf (c)\Rightarrow(b)}$.
In order to do this, we are going to show that
\begin{equation*}
 \left\|Q_T(z)\right\|_{\mathcal{L} \left(\mathcal{H}\right)}\lesssim\left\|MO^2 T^*(z)\right\|_{\mathcal{L} \left(\mathcal{H}\right)}.
\end{equation*}
Recall from above that
\begin{equation}\label{MO-Hankel}
\left\|MO^2 T^*(z)\right\|_{\mathcal{L} \left(\mathcal{H}\right)}=\sup_{\left\|e\right\|=1}\left\|H_{T^*}(k_ze)\right\|^{2}_{\mathcal F_\varphi^2}.
\end{equation}
Now, for $w\in \C^d$, we have
\begin{eqnarray}
\Vert H_{T^\ast}(K_w\,e) \Vert^2 &=& \int_{\C^d} \Vert ( T(z)^\ast - T(w)^\ast) e \Vert^2 |K(z,w)|^2\,d\mu_\varphi(z) \label{ast} \\
&\ge& \int_{D(w,a)} \Vert ( T(z)^\ast - T(w)^\ast) e \Vert^2 |K(z,w)|^2\,d\mu_\varphi(z) \nonumber
\end{eqnarray}
where $D(w,a)$ was defined in \eqref{polyball}. Lemma 7.2 and Lemma 7.1 in \cite{SY} ensure that for 
$a>0$ small enough
$$|K(z,w)|^2 \sim K(z,z)\,K(w,w)\,,\qquad w \in \C^d,\  z \in D(w,a)\,.$$
From this we deduce (as in the proof of Theorem D in \cite{SY}) that for $a>0$ small enough we have
$$|K(z,w)|^2 e^{-\varphi(z)} \gtrsim \frac{K(w,w)}{|D(w,a)|}\,,\qquad z \in D(w,a),\,$$
where  $|S|$ denotes the euclidean volume of a set $S\subset\C^d$.\\
Using the last inequality in \eqref{ast}    we obtain
\begin{equation}\label{z1}
\frac{1}{|D(w,a)|} \int_{D(w,a)} \Vert ( T(z)^\ast - T(w)^\ast) e \Vert^2 \,dm_d(z) 
\lesssim \Vert H_{T^\ast}(k_w\, e) \Vert^2\,.
\end{equation} 
We shall now show that the expression on the left-hand-side 
of the above inequality  is bounded below by a constant multiple 
of $\Vert Q_T(w)^{1/2}e\Vert^2$ for $w \in \C^d$. In order to do this, for $h \in {\mathcal H}$ with $\Vert h \Vert=1$, consider the 
holomorphic scalar-valued function
$$f(z):= f_{w,e,h}(z)=\langle (T(z)-T(w))h,\,e\rangle_{\mathcal H}\,,\qquad z \in \C^d\,.$$
Notice that
$$\Vert ( T(z)^\ast - T(w)^\ast) e \Vert = \sup_{\Vert h \Vert=1} |f(z)|\,.$$

We first prove the desired estimates for $w=(w_1,\,0,\,\cdots,\,0)$. The result in the general case will then 
follow by unitary invariance. So we first assume that $w=(w_1,\,0) \in \C^d$, where $w_1 \in \C$. 
In this case, the set $D(w,a)$ reduces to $B_1(w_1, a\rho_1(w))\times B_{d-1}(0, a\rho_2(w))$, where 
$B_k(z,R)$ denotes the euclidian ball in  $\C^k$, $k\ge 1$, centred at $z\in\C^k$ and of radius $R>0$, 
and 
\begin{equation}\label{radii}
\rho_1(z)=[\Phi'(|z|^2)]^{-1/2}, \ 
\rho_2(z)= [\Psi'(|z|^2)]^{-1/2}, \quad z\in \C^d.
\end{equation}


By Cauchy's formula,  the Cauchy-Schwarz inequality and subharmonicity, we now get
\begin{eqnarray}\label{****}
(a\rho_1(w))^2 |D_1f(w_1,0)|^2 \lesssim  \frac{1}{|D(w,a)|} \int\limits_{D(w,a)} |f(z)|^2 \,dm_d(z)\,.
\end{eqnarray}
Now notice that
$$|w|\cdot|D_1f(w)| =|\langle w_1 D_1T(w)h,\,e\rangle| =|\langle RT(w)h,e\rangle|\,.$$
In view of the above and \eqref{****}, we may write
 $$\rho_1(w)^2 \,|\langle RT(w)h,e\rangle|^2 \lesssim \frac{|w|^2}{|D(w,a)|} \int\limits_{D(w,a)} |\langle (T(z)-T(w))h,\,e\rangle | ^2\,dm_d(z)\,.$$ 
In light of relations \eqref{eigenv-comp} and \eqref{radii}, we have $\lambda(w) \asymp\rho_1(w)^{-2}$.
Taking the supremum over $\Vert h\Vert=1$ we obtain
\begin{equation}\label{R1}
\frac{1}{|w|^2} \cdot \frac{1}{\lambda(w)} \,\|(RT(w))^\ast e \|^2 \lesssim \frac{1}{|D(w,a)|} \int\limits_{D(w,a)} \Vert (T(z)^\ast-T(w)^\ast)e \Vert^2 \,dm_d(z)\, .
\end{equation}
This last estimate suffices in case $d=1$. If $d>1$,
it remains to estimate the tangential term in $Q_T(w)$. We start with the observation that since $w=(w_1,0)$, we have
$$T_{ij}(T)(w)=0\quad\text{if}\quad 1<i<j\le d \quad\text{and}\quad T_{1j}(T)(w)=\overline{w_1}\,D_jT(w)\quad\text{for}\quad 1<j\le d \,,$$
so that, in order to estimate the tangential term in $Q_T(w)$, we only need to handle terms of the form $\overline{w_1}\,D_jT(w)$. To do this we first write the Cauchy formula for the function 
$\C^{d-1} \ni z'=(z_2,\cdots,z_d) \mapsto f(z_1,z')$. We have
$$f(z_1,rz')=\int_{S_{d-1}} \frac{f(z_1,r\zeta)}{(1-\langle z',\zeta\rangle)^{d-1}}\,d\sigma(\zeta)\,,$$
where $r >0$, $S_{d-1}$ denotes the unit sphere in $\C^{d-1}$,  and $d\sigma$ is the Lebesgue measure on $S_{d-1}$. 
Differentiating with respect to $z_j$, $2 \le j \le d$, at the point $z'=0 \in \C^{d-1}$ and applying the Cauchy-Schwarz inequality, we 
deduce
$$r^2\,|D_jf(z_1,0)|^2 \lesssim \int_{S_{d-1}} |f(z_1,r\zeta)|^2\,d\sigma(\zeta)\,.$$
Using spherical coordinates in $\C^{d-1}$ together with the subharmonicity
of $z_1 \mapsto |D_jf(z_1,0)|^2$, we infer
$$(\rho_2(w))^2 |D_jf(w_1,0)|^2 \lesssim \frac{1}{|D(w,a)|}\int\limits_{D(w,a)} |f(z)|^2\,dm_d(z)\,.$$
Making $f$ explicit now yields
$$(\rho_2(w))^2 |\langle D_jT(w_1,0)h,\,e\rangle|^2 \lesssim \frac{1}{|D(w,a)|}\int\limits_{D(w,a)} 
|\langle (T(z)-T(w))h,\,e\rangle|^2\,dm_d(z)\,.$$
As before, take now the supremum over $h \in {\mathcal H}$ with $\Vert h \Vert=1$ and use the fact that
 $\mu(z)\asymp \rho_2(z)^{-2}$ (see relations \eqref{eigenv-comp} and \eqref{radii}), to deduce
\begin{equation}\label{R2}
\frac{1}{\mu(w)} \, \Vert D_jT(w)^\ast e\Vert^2 \lesssim \frac{1}{|D(w,a)|}\int\limits_{D(w,a)} 
\| (T(z)^*-T(w)^*)e\|^2\,dm_d(z)\,.
\end{equation}
Combining \eqref{R1} and \eqref{R2} and taking into account the form of $Q_T(w)$, we obtain
\begin{eqnarray}\label{t1}
\Vert Q_T(w)^{1/2} e\Vert^2 &=& \frac{1}{|w|^2 \lambda(w)} \,\Vert RT(w)^* e\Vert^2 + \frac{1}{|w|^2 \mu(w)} \,\sum_{i<j} \Vert \Big(T_{ij}(T)(w)\Big)^* e\Vert^2 \nonumber \\
&\lesssim&  \frac{1}{|D(w,a)|} \int_{D(w,a)} \Vert (T(z)^\ast -T(w)^\ast)e\Vert^2\,dm_d(z)\,.
\end{eqnarray}
We now treat the general case, that is, we let $w\in \C^d$ be arbitrary. Denote $\tilde w=(|w|,0)\in \C^d$ and let $U$ be a the unitary transformation
of $\C^d$ that maps $\tilde w$ to $w$.
Then by unitary invariance we have
$Q_T(w)=Q_{T\circ U}(U^{-1} w)=Q_{T\circ U}(\tilde w)$.
We may now make use of relation (\ref{t1}) applied to $T\circ U$, perform  the change of variables $\zeta=U z$ and take into account the fact that 
$U(D(\tilde w,a))=D(w,a)$ to deduce that \eqref{t1} holds in general.
This last fact together with relations \eqref{m-osc} and (\ref{z1}) leads us to 
\begin{equation}\label{kern-ev}
\Vert Q_T(w)^{1/2} e\Vert \lesssim \|H_{T^*} (k_w e)\|= \| \left(MO^2 T^*(w)\right)^{1/2}e \|_{\mathcal{H}}
 ,\quad {w\in \C^d},\, e\in\mathcal{H}, \|e\|_\mathcal{H}=1.
\end{equation}
Thus 
 \begin{equation}\label{Bloch-BMO}
  \Vert Q_T(w)^{1/2} \Vert_{\mathcal{L} \left(\mathcal{H}\right)}\lesssim \| MO^2 T^*(w) \|^{1/2}_{\mathcal{L} \left(\mathcal{H}\right)}
 \end{equation}
and, with this, 
our proof is complete.
\end{proof}

\begin{remark}
As already mentioned in the introduction,
the inequalities in
\eqref{m4} will be crucial in the characterizations of compactness and  Schatten class membership of the Hankel operator.
\end{remark}

\begin{corollary}
For a holomorphic function $T:\C^d\rightarrow \mathcal{L} \left(\mathcal{H}\right)$, we define
\begin{equation}\label{norm-berg}
	\left\|T\right\|_{Berg\left(\mathcal{L} \left(\mathcal{H}\right)\right)}:=\sup_{z,w\in\C^d}\frac{\left\|T(z)-T(w)\right\|_{\mathcal{L} \left(\mathcal{H}\right)}}{d_{\Psi}(z,w)}+\left\|T(0)\right\|_{\mathcal{L} \left(\mathcal{H}\right)},
\end{equation}
Then $\left\|\cdot\right\|_{Berg\left(\mathcal{L} \left(\mathcal{H}\right)\right)}$ is an equivalent norm on ${ \mathcal{B}\left(\mathcal{L} \left(\mathcal{H}\right)\right)}$.
\end{corollary}

\begin{proof} Let $T\in{ \mathcal{B}\left(\mathcal{L} \left(\mathcal{H}\right)\right)}$.
Proposition \ref{schur} $(b)$ immediately gives
 $$
 \left\|T\right\|_{Berg\left(\mathcal{L} \left(\mathcal{H}\right)\right)}\lesssim\|T\|_{{ \mathcal{B}\left(\mathcal{L} \left(\mathcal{H}\right)\right)}}
$$
On the other hand, as in Section 5 in \cite{SY}, we have
$$ \left\|H_{T^*}f(z)\right\|\leq\sup_{z,w\in\C^d} \frac{\|T(z)-T(w)\|_{ \mathcal{L} (\mathcal{ H})}}{d_{\Psi}(z,w)}Af(z),$$
where the sublinear operator $A$ defined as
$$ Af(z):=\int_{\C^d}d_{\Psi}(z,w)\left|K_{\Psi}(z,w)\right|\left\|f(w)\right\|d\mu_{\varphi}(w),\ z\in\C^d  $$
is bounded on $L^{2}\left(d\mu_{\varphi}\right)$. Therefore 
$$ \left\|H_{T^*}\right\|\leq \sup_{z,w\in\C^d} \frac{\|T(z)-T(w)\|_{ \mathcal{L} (\mathcal{ H})}}{d_{\Psi}(z,w)}\left\|A\right\|, $$
which, together with  Theorem \ref{boundedness} shows that
$$
\|T\|_{{ \mathcal{B}\left(\mathcal{L} \left(\mathcal{H}\right)\right)}}\lesssim \left\|T\right\|_{Berg\left(\mathcal{L} \left(\mathcal{H}\right)\right)}.
$$
\end{proof}

\section{Compactness of Hankel operators}

 Recall that ${\mathcal K}({\mathcal H})$ stands for the space of compact linear operators on $\mathcal{H}$.
\begin{lemma}\label{truncation}
Given $S: {\mathbb C}^d \to {\mathcal K}({\mathcal H})$ holomorphic and $R>0$, the operator 
$M_{S^\ast}^R: {\mathcal F}_\varphi^2({\mathcal H}) \to L_\varphi^2({\mathcal H})$ defined by
$$M_{S^\ast}^R f(z)=\chi_{\{\xi:\ |\xi| \le R\}}(z)\,S^\ast(z)f(z)\,,\qquad z \in {\mathbb C}^d\,,\quad f \in 
{\mathcal F}_\varphi^2({\mathcal H})\,,$$
is compact.
\end{lemma}

\begin{proof}
The proof relies on standard arguments.  For $N \in {\mathbb N}$, let $P_NS$ denote the Taylor polynomial of $S$
$$P_NS(z)
=\sum_{|\nu| \le N} K_\nu \cdot z^\nu\,,$$
where $\nu=(\nu_1,\cdots,\nu_d)\in\N^d$ and
$K_\nu$ are compact operators. 
Since 
$$
\lim_{N\rightarrow\infty} \|M_{S^\ast}^R-M_{(P_NS)^\ast}^R\|\le  \lim_{N\rightarrow\infty} \sup_{ |z| \le R} \|S(z)-P_NS(z)\|=0,
$$
in order to conclude, it is enough to show that $M_{S^\ast}^R$ is compact for $S(z)=z^\nu F$, where $F\in{\mathcal K}({\mathcal H})$ and  $\nu\in \N^d$. Moreover, since $F$ can be approximated in the operator norm by finite rank operators, we may assume that $F$ has finite
rank.  This last claim is a straightforward consequence of Montel's theorem together with relation \eqref{pointev}.


\end{proof}

\begin{theorem}\label{compactness}
Given an holomorphic function $T:\C^d\rightarrow {\mathcal L}({\mathcal H})$, the following are equivalent:
\begin{enumerate}
\item[$(a)$]  $T \in {\mathcal T}_\varphi({\mathcal L}({\mathcal H}))$ and the  Hankel operator $H_{T^\ast}$ 
is compact;
\item [$(b)$]$ T-T(0) \in {\mathcal B}_0\left({\mathcal K}({\mathcal H})\right)$;
\item [$(c)$]   $T-T(0):\C^d\rightarrow {\mathcal K}({\mathcal H})$ and $\displaystyle{\lim_{|z|\rightarrow\infty}\left\|MO^2 T^*(z)\right\|_{\mathcal{L} \left(\mathcal{H}\right)}=0}$.
\end{enumerate}
\end{theorem}

\begin{proof}
{\bf Implication } ${\bf (b)\Rightarrow(a)}$. Given $f \in {\mathcal F}_\varphi^2({\mathcal H})$ and $R>0$, by relation \eqref{m4}, we have
\begin{eqnarray*}
\Vert H_{T^\ast} f \Vert^2 &\lesssim& \sum_{j=1}^d \int_{{\mathbb C}^d} \Vert C_j(z)^\ast f(z)\Vert^2 d\mu_\varphi(z) \\
&=& \sum_{j=1}^d \int_{|z| \le R} \Vert C_j(z)^\ast f(z) \Vert^2 d\mu_\varphi(z) 
+ \int_{|z| > R} \Vert \Big( \sum_{j=1}^d C_j(z)C_j(z)^\ast\Big)^{1/2}f(z)\Vert^2 d\mu_\varphi(z) \\
&\lesssim& \sum_{k=1}^d \int_{|z| \le R}  \Vert \Big( D_kT(z)\Big)^\ast f(z)\Vert^2 d\mu_\varphi(z) + 
\sup_{|z| \ge R} \Vert Q_T(z)\Vert_{{\mathcal L}({\mathcal H})} \cdot \Vert f \Vert_\varphi^2\,,
\end{eqnarray*}
where the last step above follows from the definition of $C_j$ (see relation \eqref{cj}) as well as from 
\eqref{QT-Cj}. Now let $\varepsilon>0$ be arbitrary and choose $R>0$ such that 
$\sup\limits_{|z|>R} \Vert Q_T(z)\Vert_{{\mathcal L}({\mathcal H})}  < \varepsilon$. Then
$$\Vert H_{T^\ast} f \Vert^2 \lesssim \sum_{k=1}^d \Vert M_{(D_kT)^\ast}^R f \Vert_\varphi^2 +\varepsilon 
\Vert f \Vert_\varphi^2\,,$$
where the operators $M_{(D_kT)^\ast}^R: {\mathcal F}_\varphi^2({\mathcal H}) \to L^2_\varphi({\mathcal H})$ are compact by 
Lemma \ref{truncation}.
The above relation clearly shows that $H_{T^\ast}$ is compact.


{\bf Implication} ${\bf (a)\Rightarrow (c)}$. Assume $H_{T*}$ is compact. 
We begin by showing that $T(z)\in\mathcal{K}(\mathcal{H})$. 
For any fixed $z\in\C^d$, define the operator $N(z):{\mathcal H}\rightarrow  L_\varphi^2(\mathcal{H})$ by
 \begin{equation}\label{def N}
	 N(z)e:=H_{T^*}(k_ze),\ e\in \mathcal H.
 \end{equation}
Since, for any fixed $z \in {\mathbb C}^d$ and any sequence 
$\{e_n\}_{n \ge 1}$ which converges weakly to $0$ in ${\mathcal H}$ we obviously have that $\{k_ze_n\}_{n\ge 1}$ converges 
weakly to $0$ in ${\mathcal F}_\varphi^2({\mathcal H})$,  the compactness of $H_{T^*}$ implies that $N(z)$ is compact.
From relation (\ref{kern-ev}) we have
\begin{equation}\label{number}
\langle Q_T(z)e,e\rangle \Vert \lesssim \Vert H_{T^\ast} (k_ze) \Vert^2\,,\quad  e \in {\mathcal H},\,z \in {\mathbb C}^d.
\end{equation}
Taking into account the definition of $Q_T(z)$, this implies
\begin{equation}\label{number2}
\Vert RT(z)^*e \Vert^2 \lesssim \left|z\right|^2\lambda(z)\cdot \Vert H_{T^\ast} (k_ze) \Vert^2=\left|z\right|^2\lambda(z)\cdot \| N(z)e\|^2\,,\quad e\in\mathcal{H},\, z \in {\mathbb C}^d.
\end{equation}
The compactness of $N(z)$ now ensures that $RT(z)^*$ and hence $RT(z)$
is compact for any $z \in {\mathbb C}^d$.
 Then 
$$T(z)-T(0)=\int_0^1 \frac{1}{t}\,RT(tz)\,dt$$
implies that $T(z)-T(0)$ is compact for any $z \in {\mathbb C}^d$. 
It remains to show that 
\begin{equation}\label{number3}
\lim_{|z| \to \infty} \Vert MO^2 T^*(z)  \Vert_{{\mathcal L}({\mathcal H})}  =0\,. 
\end{equation}
Since, for any fixed $z\in\C^d$, $N(z)$ is a compact operator on ${\mathcal H}$, it attains its norm, i.e. there 
exists $e_z \in {\mathcal H}$ with $\Vert e_z \Vert=1$ such that 
$$\left\|H_{T^*}(k_ze_z)\right\|_{\mathcal F_\varphi^2}= \left\|N(z)e_z\right\|_{\mathcal H}=\left\|N(z)\right\|_{{\mathcal L}({\mathcal H})}=  \left\| MO^2 T^*(z)\right\|^{1/2}_{\mathcal L({\mathcal H})},$$
where the last equality above follows from \eqref{m-osc} and \eqref{def N}. 
 Hence,  \eqref{number3} will immediately follow, once we show that
$\{k_ze_z\}$ converges weakly to $0$ in ${\mathcal F}_\varphi^2({\mathcal H})$ as $|z| \to \infty$. Indeed, 
for 
any holomorphic polynomial $f$ with coefficients in $\mathcal{H}$, we have
\begin{eqnarray*}
|\langle f,k_ze_z\rangle| &=& \Big| \int_{{\mathbb C}^d} \langle f(\xi),e_z\rangle\,\overline{k_z(\xi)}\,d\mu_\varphi(\xi)\Big| \\
&=& \frac{|\langle f(z),e_z\rangle|}{\Vert K_z \Vert} \to 0\quad\hbox{as}\quad |z| \to \infty,
\end{eqnarray*}
where the last step follows from the estimate $\Vert K_z \Vert\asymp e^{\Psi(|z|^2)/2} \Phi'(|z|^2)^{1/2}(\Psi'(|z|^2))^{(d-1)/2}$ (see \cite{SY}). The assertion for a general $f \in {\mathcal F}_\varphi^2({\mathcal H})$ is easily deduced from the above by approximation with polynomials.

In order to conclude, it is enough to prove that $(c)\Rightarrow(b)$, but this is a direct consequence of relation  \eqref{Bloch-BMO}.
\end{proof}

\section{Schatten classes}
The aim of this section is to characterize the Schatten class membership of $H_{T^*}$. 
We begin with an identity which we are going to formulate on  Fock spaces, although its analogue holds for a large 
class of vector-valued spaces of analytic functions.
\begin{lemma}\label{trace2} 
 Let $S$ be a positive operator on ${\mathcal F}_\varphi^2({\mathcal H})$,
and, for each fixed $z\in \C^d$, let $\{e_k^z\}_{k\ge 1}$ be an orthonormal basis of $\mathcal{H}$ 
(possibly) depending on $z$. 

\noindent Then $\sum_{k\ge 1}   \langle \, S(K_z e_k^z )\,,\, K_z e_k^z\,\rangle_{{\mathcal F}_\varphi^2({\mathcal H})}$ is independent of the choice of $\{e_k^z\}_{k\ge 1}$ and we have
\begin{equation}\label{trace}
trace(S)=\int_{\C^d}\sum_{k\ge 1}   \langle \, S(K_z e_k^z )\,,\, K_z e_k^z\,\rangle_{{\mathcal F}_\varphi^2({\mathcal H})}\, d\mu_\varphi(z),
\end{equation}
where $K_z$ denotes the reproducing kernel of ${\mathcal F}_\varphi^2(\C)$ at the point $z\in\C^d$.
\end{lemma}
\begin{proof}
If $\{E_n\}_{n\ge 1}$ is an orthonormal basis of ${\mathcal F}_\varphi^2({\mathcal H})$, then 
\begin{eqnarray}\label{trace1}
trace(S)
=\sum_{n\ge 1} \|S^{1/2}E_n\|^2=\sum_{n\ge 1} \int_{\C^d} \|(S^{1/2}E_n)(z)\|^2\, d\mu_\varphi(z).
\end{eqnarray}
We now have
\begin{eqnarray*}
\sum_{n\ge 1}  \|(S^{1/2}E_n)(z)\|^2&=&
\sum_{n\ge 1} \sum_{k\ge 1} |\langle (S^{1/2}E_n)(z), e_k^z \rangle|^2\\
&=&\sum_{n\ge 1} \sum_{k\ge 1} \left| \int_{\C^d}    \langle (S^{1/2}E_n)(\zeta), e_k^z \rangle\, \overline{K_z(\zeta)}\, d\mu_\varphi(\zeta)   \right|^2\\
&=&\sum_{n\ge 1} \sum_{k\ge 1} |\langle S^{1/2}E_n, K_z e_k^z \rangle_{{\mathcal F}_\varphi^2({\mathcal H})}|^2
=\sum_{k\ge 1}  \sum_{n\ge 1} |\langle E_n, S^{1/2}(K_z e_k^z )\rangle_{{\mathcal F}_\varphi^2({\mathcal H})}|^2\\
&=&\sum_{k\ge 1}   \| S^{1/2}(K_z e_k^z )\|^2_{{\mathcal F}_\varphi^2({\mathcal H})}
=\sum_{k\ge 1}   \langle \, S(K_z e_k^z )\,,\, K_z e_k^z\,\rangle_{{\mathcal F}_\varphi^2({\mathcal H})},
\end{eqnarray*}
where the second equality above follows by the reproducing formula. The last relation together with \eqref{trace1} leads us now to \eqref{trace}.
\end{proof}

We now turn to the proof of Theorem \ref{Schatten}.
\begin{proof}[Proof of Theorem \ref{Schatten}.] 
\ 
\\

{\bf Implication} ${\bf (a)\Rightarrow(c)}$. Assume $H_{T^*}\in\mathcal{S}^p.
$
If $\{e_k\}_{k\ge 1}$ is an arbitrary orthonormal basis of $\mathcal{H}$, in view of \eqref{m-osc} we obtain 
$$
\sum_{k\ge 1} \|\left(MO^2 T^*(z)\right)^{1/2} e_k\|^p = \sum_{k\ge 1}  \|H_{T^*}(k_z e_k)\|^p<\infty, \quad z\in \C^d,
$$
since $\{k_z e_k\}_{k\ge 1}$ is an orthonormal set in ${\mathcal F}_\varphi^2({\mathcal H})$ and $p\ge 2$. This shows that $\left(MO^2 T^*(z)\right)^{1/2} \in{\mathcal{S}^p}(\mathcal{H})$.
Now, since $\mathcal{S}^p(\mathcal{H})$ is separable (see \cite{Gohberg}, Chapter 3, Section 6), by Pettis' theorem \cite{tal}, in order to show that $z\mapsto \left(MO^2 T^*(z)\right)^{1/2}$ is measurable, it suffices to prove that it is 
weakly measurable. With this aim, let $S\in \mathcal{S}^q(\mathcal{H})$, where $1/p+1/q=1$.
If $\{e_k\}_{k\ge 1}$ is an orthonormal basis of $\mathcal{H}$, we have 
\begin{eqnarray*}
\langle \left(MO^2 T^*(z)\right)^{1/2}, S\rangle&=&trace({S^*\left(MO^2 T^*(z)\right)^{1/2}})\\
&=&\sum_{k\ge 1}\langle \left(MO^2 T^*(z)\right)^{1/2} e_k,S e_k\rangle,\quad z\in \C^d,
\end{eqnarray*} 
and, since the last expression above defines a measurable function, it follows that $\left(MO^2 T^*(z)\right)^{1/2}$ is measurable.

In order to prove \eqref{SCH-BMO}, let $\{e_n^z\}_{n\ge 1}$  be an orthonormal basis of $\mathcal{H}$ that diagonalizes the compact self-adjoint operator $\left(MO^2 T^*(z)\right)^{1/2}$. Apply Lemma \ref{trace2} to $S:=((H_{T^*})^*H_{T^*})^{p/2}$ to deduce
\begin{eqnarray}\label{htsp}
\|H_{T^*}\|_{\mathcal{S}^p}=trace[((H_{T^*})^*H_{T^*})^{p/2}]=\int_{\C^d}\sum_{k\ge 1}   \big \langle \, ((H_{T^*})^*H_{T^*})^{p/2}(K_z e_k^z ), K_z e_k^z\,\big\rangle d\mu_\varphi(z).
\end{eqnarray}
For each $z\in\C^d$, in view of Jensen's inequality, relation \eqref{m-osc}  and the choice of $e_k^z$, we obtain
\begin{eqnarray*}
\sum_{k\ge 1}   \Big\langle \, ((H_{T^*})^*H_{T^*})^{p/2}(K_z e_k^z )\,,\, K_z e_k^z\,\Big\rangle
&\ge& \sum_{k\ge 1}  \Big \langle \, (H_{T^*})^*H_{T^*}(k_z e_k^z )\,,\, k_z e_k^z\,\Big\rangle^{p/2} K(z,z)\\
&=& \sum_{k\ge 1}  \| H_{T^*}(k_z e_k^z )\|^p K(z,z)\\
&=&\sum_{k\ge 1}    \langle MO^2 T^*(z) e_k^z , e_k^z\rangle^{p/2}K(z,z) \\
&=&     \| \left(MO^2 T^*(z)\right)^{1/2}\|_{\mathcal{S}^p(\mathcal{H})}^{p}\, K(z,z), 
\end{eqnarray*}
and, returning to \eqref{htsp}, we get   $(c)$.

{\bf Implication} ${\bf (c)\Rightarrow(b)}$. This a direct consequence of relation \eqref{kern-ev}. Indeed,
for any  orthonormal basis $(e_k)_k$ of $\mathcal H$, we have 
 \begin{align*}
\sum_{k\ge 1} \|Q_T(z)^{1/2} e_k\|^p &\lesssim \sum_{k\ge 1}\left\langle MO^2 T^*(z)e_k,e_k\right\rangle^{p/2}.
\end{align*}
Since $p\ge 2$, this implies $ \|Q_T(z)^{1/2}\|^p_{\mathcal S^p(\mathcal H)}\lesssim \left\|\left(MO^2 T^*(z)\right)^{1/2}\right\|^{p}_{\mathcal S^p(\mathcal H)}$, and  thus $(b)$ follows. 

{\bf Implication} ${\bf (b)\Rightarrow(a)}$. Recall that, by \eqref{m4}, we have
\begin{equation}\label{m4n}
\|H_{T^*} f\|\lesssim \int_{\C^d} \|Q_T(z)^{1/2} f(z)\|^2 \, d\mu_\varphi(z),\quad f\in {\mathcal F}_\varphi^2({\mathcal H}).
\end{equation}
Hence,  if  the multiplication operator 
$$
M_{Q^{1/2}} f(z):= Q_T(z)^{1/2} f(z),\quad  f\in {\mathcal F}_\varphi^2({\mathcal H}),\, z\in\C^d,
$$
belongs to some Schatten class  ideal, then $H_{T^*} $ will have the same property.

We now provide a sufficient condition for  Schatten class membership of multiplication operators using an interpolation argument.
For a strongly-measurable
operator-valued function $R:\C^d\rightarrow \mathcal{L}(\mathcal{H})$, consider the operator
$$
M_{R} f(z):= R(z) f(z),\quad  f\in {\mathcal F}_\varphi^2({\mathcal H}), \, z\in\C^d.
$$
For $p\ge 2$, we denote by $L^p(\C^d,\mathcal{S}^p({\mathcal H}), d\lambda_\varphi)$ the space of strongly measurable functions
$g:\C^d\rightarrow \mathcal{S}^p({\mathcal H})$ satisfying
$$
\int_{\C^d} \|g(z)\|^p_{  \mathcal{S}^p(\mathcal{H})} \, d\lambda_\varphi(z)<\infty.
$$
Moreover, $L^\infty(\C^d,\mathcal{L}({\mathcal H}), d\lambda_\varphi)$ will stand for the closure in the supremum 
norm of $\mathcal{L}({\mathcal H})-$valued simple functions (see \cite{berg}, Ch. 5, page 107). 

We claim that, if $R\in L^2(\C^d,\mathcal{S}^2({\mathcal H}), d\lambda_\varphi)$, then $M_R: {\mathcal F}_\varphi^2(\mathcal{H}) \rightarrow { L}_\varphi^2(\mathcal{H}) $ is a Hilbert-Schmidt
operator. To this end, let $\{e_n\}_{n\ge 1}$ be an orthonormal basis of the scalar Fock space ${\mathcal F}_\varphi^2(\C)$
and let $\{f_k\}_{k\ge 1}$ be an orthonormal basis of $\mathcal{H}$. Then it is clear that $\{E_{n,k}(z):=e_n(z)f_k\}_{n,k\ge1}$ 
is an orthonormal basis of ${\mathcal F}_\varphi^2({\mathcal H})$, and we have
\begin{eqnarray}\label{inter}
\|M_R\|^2_{\mathcal{S}^2}&=&\sum_{n,k\ge 1}\|M_R (E_{n,k})\|^2=\sum_{n,k\ge 1} \int_{\C^d}\| R(z) f_k\|^2 \, |e_n(z)|^2 \,d\mu_\varphi(z)\nonumber\\
&=&
\sum_{n\ge 1} \int_{\C^d}\| R(z)\|_{\mathcal{S}^2(\mathcal{H})}^2 \, |e_n(z)|^2 \,d\mu_\varphi(z)\nonumber\\
&=& \int_{\C^d}\| R(z)\|_{\mathcal{S}^2(\mathcal{H})}^2 \, K(z,z) \,d\mu_\varphi(z)=\|R\|^2_{ L^2(\C^d,\mathcal{S}^2({\mathcal H}), d\lambda_\varphi)},
\end{eqnarray}
and the claim follows. 

Moreover,  if $R\in L^\infty(\C^d,\mathcal{L}({\mathcal H}), d\lambda_\varphi)$, we have
\begin{equation}\label{inter1}
\|M_R \|_{\mathcal{S}^\infty}\le \hbox{ess\,sup}_{z\in\C^d} \|R(z)\|_{\mathcal{L}(\mathcal{H})},
\end{equation}
where $\mathcal{S}^\infty$ denotes the space of bounded linear operators from ${\mathcal F}_\varphi^2(\mathcal{H})$ to $L_\varphi^2(\mathcal{H})$.
Taking into account \eqref{inter}, \eqref{inter1} together with Theorem 5.1.2 (page 107) in \cite{berg}, it now follows by interpolation 
that
$$
\|M_R\|_{\mathcal{S}^p}\le \|R\|_{ L^p(\C^d,\mathcal{S}^p({\mathcal H}), d\lambda_\varphi)}, \quad p\ge 2.
$$
Particularizing $R(z):=Q_T(z)^{1/2}$ above and using \eqref{m4n} now yield
$$
\|H_{T^*}\|^p_{\mathcal{S}^p}\lesssim \|M_{Q_T^{1/2}}\|^p_{\mathcal{S}^p}\le 
\int_{\C^d} \|Q_T(z)^{1/2}\|^p_{  \mathcal{S}^p(\mathcal{H})} K(z,z)\, d\mu_\varphi(z),
$$
which completes the proof.
\end{proof}


{\it Acknowledgements.} We would like to thank N. Nikol'skii for suggesting to us the method of proof of Theorem \ref{density pol}.
\bibliographystyle{plain}

\end{document}